\documentclass[11pt, oneside]{amsart} 

\usepackage[english]{babel} 
\usepackage{graphics,color,pgf}
\usepackage{epsfig}
\usepackage[ansinew]{inputenc}
\usepackage[all]{xy}
\usepackage{tikz-cd}
\usepackage{bbm}
\usetikzlibrary{matrix,arrows,decorations.pathmorphing}
\newdir{ >}{!/8pt/@{}*@{>}}
\usepackage{amssymb, amsmath,amsthm, mathtools, amscd}
\usepackage{mathrsfs}
\usepackage[margin=1.5in]{geometry}
\usepackage{tikz}
\usetikzlibrary{automata, arrows.meta, positioning}
\usepackage[pagebackref, pdfborder={0 0 0}
  ,urlcolor=black,hypertexnames=false]{hyperref}

\usepackage[margin=1.5in]{geometry}
  
\makeatletter
\newtheorem*{rep@theorem}{\rep@title}
\newcommand{\newreptheorem}[2]{%
\newenvironment{rep#1}[1]{%
 \def\rep@title{#2 \ref{##1}}%
 \begin{rep@theorem}}%
 {\end{rep@theorem}}}
\makeatother

\newtheorem{intro_thm}{Theorem}

\newtheorem{lemma}{Lemma}[section]
 
\newtheorem{prop}[lemma]{Proposition}
\newtheorem{cor}[lemma]{Corollary}

\theoremstyle{definition}
\newtheorem{defn}[lemma]{Definition}
\newtheorem{es}[lemma]{Example}

\theoremstyle{remark}
\newtheorem{oss}[lemma]{Remark}
\newtheorem{nota}[lemma]{Notation}

\newtheoremstyle{TheoremNum}
        {0.2 cm}{0.2 cm}              
        {\itshape}                      
        {}                              
        {}                     
        {.}                             
        { }                             
        {\thmname{\bfseries #1}\thmnote{ \bfseries #3}}
    \theoremstyle{TheoremNum}
   
\newtheorem{rec_thm}{Theorem}

\newtheorem{claim*}[rec_thm]{Claim}

\newcommand{\id}{\mathrm{id}}

\begin{document}

\title[Superrigidity for transverse measured groupoids]{Similarity and superrigidity for transverse measured groupoids}

\author[F. Sarti]{F. Sarti}
\address{Department of Mathematics, University of Pisa, 
Pisa, Italy}
\email{filippo.sarti@unipi.it}

\author[A. Savini]{A. Savini}
\address{Department of Mathematics, University of Milano-Bicocca, Milan, Italy}
\email{alessio.savini@unimib.it}

\date{\today.\ \copyright{\ F. Sarti, A. Savini}.}

\begin{abstract}
Let $G$ be a locally compact, second countable, unimodular group and let $(X,\mu,Y)$ be an ergodic integrable transverse $G$-system. We show that $G \ltimes X$ endowed with the product measure and the restriction $(G \ltimes X)|_Y$ endowed with the transverse measure are similar ergodic groupoids. As a consequence, we prove a superrigidity phenomenon for higher rank transverse groupoids. More precisely, for $i=1,\ldots,k$, let $\mathbf{G}_i$ be a connected, simply connected, semisimple algebraic group over some local field $\kappa_i$ of characteristic zero. Let $G_i=\mathbf{G}_i(\kappa_i)$ be the $\kappa_i$-points of $\mathbf{G}_i$ and denote by $G=\prod_{i=1}^k G_i$. If we assume that $G$ has higher rank and each factor has positive rank, given an ergodic integrable transverse $G$-system $(X,\mu,Y)$, we prove that any unbounded Zariski dense morphism of the transverse groupoid $(G \ltimes X)|_Y$ into either an almost simple or a reductive algebraic group is superrigid.
\end{abstract}
\maketitle

\section{Introduction}

A subgroup $\Gamma$ in a semisimple Lie group $G$ with finite center and no compact factors is a \emph{lattice} if it is discrete and it has finite covolume with respect to the Haar measure. In \cite{margulis:super} Margulis proved that every unbounded Zariski dense representation of a higher rank irreducible lattice $\Gamma$ into an adjoint semisimple Lie group can be actually extended to the ambient group $G$. This phenomenon, called \emph{superrigidity}, was crucial for Margulis to prove his \emph{arithmeticity} statement in the higher rank setting. 

Later on, Zimmer \cite{zimmer:annals} showed an analogous result for Zariski dense \emph{measurable cocycles} associated to $\Gamma$. Roughly speaking, a measurable cocycle is a twisted representation where the twist depends on some Lebesgue $\Gamma$-space $(X,\mu)$. Zimmer proved that every unbounded Zariski dense measurable cocycle is \emph{superrigid}, namely it can be untwisted to an actual representation of $\Gamma$ (or even of $G$). In his book \cite[Chapter 5]{zimmer:libro}, he pointed out that one can recover Margulis superrigidity from his theorem by fixing as parameter space $X=G/\Gamma$. More precisely, given a measurable section $\theta:G/\Gamma \rightarrow G$ of the canonical projection, one can define a measurable cocycle by setting
$$
\sigma_\theta:G \times G/\Gamma \rightarrow \Gamma, \ \ \sigma_\theta(g,x):=\theta(gx)g\theta(x)^{-1}. 
$$
Additionally, different choices of sections produce cohomologous cocycles. Any unbounded Zariski dense representation $\rho:\Gamma \rightarrow H$ into an adjoint semisimple Lie group can be composed with $\sigma_\theta$, leading in this way to the Zariski dense measurable cocycle $\rho \circ \sigma_\theta$ on which Zimmer superrigidity applies. 

\subsection{Transverse measured groupoids and similarity} The above picture can be translated into the language of measured groupoids. A \emph{groupoid} is a small category in which every morphism is invertible. The space of objects $\mathcal{G}^{(0)}$ is the set of \emph{units} and we denote by $s,t$ the \emph{source} and \emph{target} maps, respectively. A \emph{Haar system} for the target map $\lambda=\{\lambda^x\}_{x \in \mathcal{G}^{(0)}}$ is a family of measures which is equivariant with respect to the groupoid action and where each $\lambda^x$ is supported in $t^{-1}(x)$. A groupoid $\mathcal{G}$ is \emph{measured} if there exists a measure $\nu$ on $\mathcal{G}^{(0)}$ such that the measure class of the convolution $\nu \circ \lambda$ is fixed by the inverse map (see Section \ref{secgroupoids}). The measured groupoid $(\mathcal{G},\lambda,\nu)$ is \emph{ergodic} if any $\mathcal{G}$-invariant Borel subset of $\mathcal{G}^{(0)}$ has either full or null $\nu$-measure. 

Let $G$ be a locally compact, second countable, unimodular group acting on a standard probability space $(X,\mu)$ via a probability measure-preserving action. A standard example of measured groupoid is given by the \emph{action groupoid} $G \ltimes X$ endowed with the product measure $m_G \otimes \mu$. In this case, the unit space is precisely $X$, the source map is the projection on the second factor and the target function boils down to the action map. The Haar measure $m_G$ has the role of (constant) Haar system. 

For any Borel subset $Y \subset X$, one can define the \emph{restricted groupoid} $(G \ltimes X)|_Y$ as those elements whose both source and target are in $Y$. We say that $Y$ is a \emph{cross section} if $GY=X$ and the hitting time set 
$$
Y_x:=\{g \in G \ | \ gx \in Y \} 
$$
is not empty and locally finite for every $x \in X$. In this setting, the restricted groupoid has countable fibers with respect to the target map. If additionally the point process $x \mapsto Y_x$ is locally integrable, Avraham-Re'em, Bj\"orklund and Cullman \cite{ABC} proved that there exists a suitable probability measure $\nu$ such that its convolution with the counting measures defines on $\mathcal{G}$ a structure of measured groupoid (see Proposition \ref{Campbell}). If we assume that the measure $\nu$ is ergodic, we call the triple $(X,\mu,Y)$ \emph{ergodic integrable transverse $G$-system} (see Section \ref{section:transverse}) and we use $(\mathcal{G},\nu)$ to refer to the associated \emph{transverse groupoid}. A crucial example is given by any lattice $\Gamma<G$. In this context, a cross section for the action groupoid $G\ltimes \Gamma\backslash G$ is given by $\{e\Gamma\} \subset G/\Gamma$ and the associated transverse groupoid boils down to the lattice itself. This suggests that transverse measured groupoids can be viewed as a natural generalization of lattices. As suggested by Hartnick (personal communication), the entire theory of \emph{approximate lattices} \cite{BH} can be actually rephrased within the framework of measured groupoids. More precisely, given an approximate lattice $\Lambda$ in $G$, one can naturally associate to $\Lambda$ a transverse groupoid known as \emph{pattern groupoid} (see Bj\"{o}rklund, Hartnick and Karasik \cite{BHK} for further details); indeed, this construction extends to the more general setting of \emph{FLC subsets}. This perspective moved Hartnick and the first named author \cite{HS25} to study transverse groupoids from point of view which is both dynamical and cohomological.

Given two measured groupoids $(\mathcal{G},\lambda,\nu)$ and $(\mathcal{H},\lambda',\nu')$, a \emph{morphism} $\varphi:\mathcal{G} \rightarrow \mathcal{H}$ is a measurable map which is an algebraic morphism on the restriction of $\mathcal{G}$ to a full measure subset of units (namely an \emph{inessential contraction}) and behaves well on negligible sets. Two morphisms $\varphi,\psi:\mathcal{G} \rightarrow \mathcal{H}$ are \emph{similar} if there exists a measurable map $h:\mathcal{G}^{(0)} \rightarrow \mathcal{H}$ whose restriction to an inessential contraction of $\mathcal{G}$ intertwines $\varphi$ and $\psi$. The starting groupoids are \emph{similar} if there exist two morphisms $\varphi:\mathcal{G} \rightarrow \mathcal{H}$ and $\psi:\mathcal{H} \rightarrow \mathcal{G}$ whose compositions are similar to the identities $\mathrm{id}_\mathcal{G}$ and $\mathrm{id}_{\mathcal{H}}$, respectively. 

The main result of the paper is the following:

\begin{intro_thm}\label{thm:restriction:similar}
Let $G$ be a locally compact, second countable, unimodular group and let $(X,\mu,Y)$ be an ergodic integrable transverse $G$-system. Let $\nu$ be the transverse measure on $Y$. Then $(G \ltimes X,m_G,\mu)$ and $(\mathcal{G},\nu)$ are similar ergodic groupoids. 
\end{intro_thm}

When $\Gamma<G$ is a lattice, the previous statement boils down to the well-known fact that $G \ltimes G/\Gamma$ and $\Gamma$ are similar.

An immediate consequence of Theorem \ref{thm:restriction:similar} is given by the following:

\begin{intro_thm}\label{thm:property:T}
Let $G$ be a locally compact, second countable, unimodular group with Kazhdan property $\mathrm{(T)}$. Let $(X,\mu,Y)$ be an ergodic integrable $G$-system and denote by $(\mathcal{G},\nu)$ the associated transverse groupoid. Let $\rho:\mathcal{G} \rightarrow H$ be a morphism into an amenable group. Then $\rho$ is similar to a morphism into a compact subgroup of $H$. 
\end{intro_thm}

\subsection{Zariski density and superrigidity} Inspired by Zimmer's observation \cite{zimmer:libro} about the possibility of deriving Margulis superrigidity from his statement about measurable cocycles and following the connection between lattices and transverse groupoids, in these notes we exploit Theorem \ref{thm:restriction:similar} to study two superrigidity phenomena occurring for \emph{Zariski dense} morphisms of transverse groupoids.
It is important to emphasize that the notion of Zariski density is well-defined in this context thanks to the framework of algebraic representability for ergodic groupoids recently introduced by the authors \cite{sarti:savini:boundaries} (see Section \ref{sec:alg:rep} for a precise definition).

Before stating the main results, we need to fix some notation. Let $S=\{p_1,\ldots,p_\ell\}$ be a finite set of integer primes, where we allow one of them to be $\infty$. For each $p_i \in S$, we denote by $\kappa_i=\mathbb{Q}_{p_i}$, with the usual convention that $\mathbb{Q}_\infty=\mathbb{R}$. Let $\mathbf{G}_i$ be a connected, simply connected, semisimple algebraic $\kappa_i$-group.
We denote by $G_i=\mathbf{G}_i(\kappa_i)$ the $\kappa_i$-points of $\mathbf{G}_i$. We set
$$
r_i:=\mathrm{rank}_{\kappa_i}(\mathbf{G}_i), \ \ \  r:=\sum_{i=1}^\ell r_i.
$$
For $G=\prod_{i=1}^\ell G_i$, we fix an ergodic integrable transverse $G$-system $(X,\mu,Y)$ and we denote by $(\mathcal{G},\nu)$ the transverse groupoid determined by it. 

Given a field $\kappa$, recall that an algebraic $\kappa$-group is \emph{almost $\kappa$-simple} if the only normal subgroups are finite, whereas it is \emph{reductive} if its unipotent radical is trivial. We have the following pair of superrigidity results for morphisms of transverse groupoids.

\begin{intro_thm}\label{thm:superrigidity}
Suppose that $r_i>0$, for $i=1,\ldots, \ell$, and $r \geq 2$. Let $\kappa$ be equal either to $\mathbb{R},\mathbb{C}$ or $\mathbb{Q}_p$ and consider $H=\mathbf{H}(\kappa)$ the $\kappa$-points of an almost $\kappa$-simple $\kappa$-algebraic group. Consider a Zariski dense morphism $\rho:\mathcal{G} \rightarrow H$. Then one of the following holds:
\begin{enumerate}
\item The morphism $\rho$ is similar to a homomorphism with values into a compact subgroup of $H$. 
\item If $\kappa=\mathbb{Q}_{p_i}$ for some $p_i \in S \setminus \{\infty\}$, then $\rho$ is similar to a morphism induced by the composition 
$$
G \rightarrow G_i \rightarrow H,
$$
where the first map is the projection on the $i$-th factor and the second one is a rational epimorphism.
\item If $\kappa=\mathbb{R},\mathbb{C}$, then $\rho$ is similar to a morphism induced by the composition
$$
G \rightarrow G_\infty \rightarrow H,
$$
where the first map is the projection on the real factor and the second one is a rational epimorphism.
\end{enumerate}
\end{intro_thm}

\begin{intro_thm}\label{thm:margulis:fisher}
Assume that $r_i \geq 2$, for $i=1,\ldots,\ell$. Given a local field of characteristic zero $\kappa$, let $\mathbf{H}$ be a reductive $\kappa$-algebraic group with $H=\mathbf{H}(\kappa)$. Every Zariski dense morphism $\rho:\mathcal{G} \rightarrow H$ is similar to another morphism induced by a globally defined representation $\pi:G \rightarrow H$, modulo a morphism $\sigma:\mathcal{G} \rightarrow K \subset H$, where $K$ is a compact subgroup centralizing $\pi(G)$. 
\end{intro_thm}

Both Theorem \ref{thm:superrigidity} and Theorem \ref{thm:margulis:fisher} are known for measurable cocycles by \cite[Theorem 10.1.6]{zimmer:libro} and by \cite[Theorem 3.16]{margulis:fisher}, respectively. 

To proofs of Theorems \ref{thm:superrigidity} and \ref{thm:margulis:fisher} revolve on the fact that a similarity between two ergodic groupoids induces a bijection on their Zariski dense morphisms (when we fix the codomain). Since $G \ltimes X$ and the restricted groupoid $\mathcal{G}$ are similar by Theorem \ref{thm:restriction:similar}, the standard theorems for the semidirect groupoid $G \ltimes X$ lead to the desired conclusion. 

%

\subsection*{Plan of the paper} In Section \ref{first:sec} we recall several notions that we will need in the paper. In Section \ref{secgroupoids} we remind the definition of measured groupoids, the theory about morphisms and similarities. In Section \ref{sec:alg:rep} we recall the notion of algebraic representability for ergodic groupoids. Section \ref{section:transverse} is devoted to integrable transverse systems. There we prove Theorems \ref{thm:restriction:similar} and \ref{thm:property:T}. In Section \ref{zariski:dense:rep} we study Zariski dense morphisms of transverse groupoids. In that section we prove Theorems \ref{thm:superrigidity} and \ref{thm:margulis:fisher}.

\section{Measured groupoids}\label{first:sec}

In this section we quickly recall all the preliminary definitions needed throughout the paper. 

\subsection{Measured groupoids and morphisms}\label{secgroupoids} A \emph{groupoid} $\mathcal{G}$ is a small category where all the morphisms are invertible. We denote the set of objects by $\mathcal{G}^{(0)}$ and the set of morphisms by $\mathcal{G}^{(1)}$. We refer to the natural \emph{source} and \emph{target} maps using the notation $s:\mathcal{G}^{(1)} \rightarrow \mathcal{G}^{(0)}$ and $t:\mathcal{G}^{(1)} \rightarrow \mathcal{G}^{(0)}$, respectively. Given two objects $x,y \in \mathcal{G}^{(0)}$, we set
$$
\mathcal{G}_x:=\{ g \in \mathcal{G}^{(1)} \ | \ s(g)=x\}, \ \ \ \mathcal{G}^y:=\{ g \in \mathcal{G}^{(1)} \ | \ t(g)=y\}. 
$$
The natural embedding of $\mathcal{G}^{(0)}$ into $\mathcal{G}^{(1)}$ via the map $x \rightarrow 1_x$ allows us to write $s(g)=g^{-1}g$ and $t(g)=gg^{-1}$. A subset $U \subset \mathcal{G}^{(0)}$ is \emph{invariant} if, for every $g \in \mathcal{G}^{(1)}$, we have that $s(g) \in U$ if and only if $t(g) \in U$. Given any subset $U \subset \mathcal{G}^{(0)}$ of objects, we define the \emph{restricted groupoid} $\mathcal{G}|_U$ as the groupoid with objects $U$ and morphisms
$$
\mathcal{G}|_U^{(1)}:=\{ g \in \mathcal{G}^{(1)} \ | \ s(g),t(g) \in U\}. 
$$
The \emph{saturation} of $U$ is the subset
$$
\mathcal{G}U:=\{t(g) \ | \ s(g) \in U\}. 
$$
A groupoid $\mathcal{G}$ is \emph{Borel} if both $\mathcal{G}^{(1)}$ and $\mathcal{G}^{(0)}$ are standard Borel spaces and both the composition and the inverse map are Borel. 
A \emph{Borel system} of measures for the target map $t:\mathcal{G}^{(1)} \rightarrow \mathcal{G}^{(0)}$ is a family $\lambda=\{\lambda^x\}_{x \in \mathcal{G}^{(0)}}$ of measures on $\mathcal{G}^{(1)}$ such that $\lambda^x(\mathcal{G}^{(1)} \setminus \mathcal{G}^x)=0$ and the map
$$
x \mapsto \lambda^x(f):=\int_{\mathcal{G}^{(1)}} f(g)d\lambda^x(g)
$$
is Borel, for every bounded Borel function $f:\mathcal{G}^{(1)} \rightarrow \mathbb{R}$. We call $\lambda$ a \emph{Haar system} if it is a Borel system for the target map and it satisfies
$$
g_\ast \lambda^{s(g)}=\lambda^{t(g)}, 
$$
for every $g \in \mathcal{G}^{(1)}$. A Borel groupoid $\mathcal{G}$ is $t$-\emph{discrete} if the fibers of the target map are countable. In that case, a natural Haar system is given by the counting measures on $t$-fibers (and hence we will omit it in the notation). 

Given a Haar system $\lambda$, any probability measure $\nu$ on $\mathcal{G}^{(0)}$ defines a measure on $\mathcal{G}^{(1)}$ as follows: given $E \subset \mathcal{G}^{(1)}$, we define
$$
(\nu \circ \lambda)(E):=\int_{\mathcal{G}^{(0)}} \lambda^x(E) d\nu(x). 
$$
We say that $\nu$ is \emph{invariant} (respectively \emph{quasi-invariant}) if the measure $\nu \circ \lambda$ (respectively, its measure class) is invariant under the inverse map $g \mapsto g^{-1}$. Given a quasi-invariant probability measure $\nu$ on $\mathcal{G}^{(0)}$, we call the triple $(\mathcal{G},\lambda,\nu)$ a \emph{measured groupoid}. We say that the groupoid is \emph{ergodic} if for any invariant subset $U \subset \mathcal{G}^{(0)}$ it holds either $\nu(U)=0$ or $\nu(\mathcal{G}^{(0)} \setminus U)=0$.

\begin{es}[Action groupoid]\label{es:action:groupoid}
Let $G$ be a locally compact, second countable, unimodular group and let $(X,\mu)$ be a standard probability space with a measure-preserving $G$-action. The product $\mathcal{G}:=G \times X$ has a natural structure of measured groupoid. The composable elements are given by 
$$
\{ ((g_1,x_1),(g_2,x_2)) \in \mathcal{G}^2 \ | \ x_1=g_2x_2 \}
$$
and the groupoid laws are defined as 
$$
(g_1,g_2x)(g_2,x):=(g_1g_2,x), \ \ \ (g,x)^{-1}:=(g^{-1},gx). 
$$
In this context, the unit space is $X$ with source given by the projection on the second factor and target given by the action map. If we denote by $m_G$ the Haar measure on $G$, then the constant family $\{m_G\}_{x \in X}$ is naturally a Haar system for the target function and $\mathcal{G}$ becomes a measured groupoid when endowed with the measure $m_G \otimes \mu$. Additionally, this groupoid is ergodic if and only if the $G$-action on $(X,\mu)$ is so. We call $(\mathcal{G},m_G, \mu)$ \emph{action groupoid} and we denote it by $G \ltimes X$. 
\end{es}

\begin{es}[Restricted groupoid to a positive measure subset]\label{es:res:pos:groupoid}
Let $(\mathcal{G},\lambda,\nu)$ be a measured groupoid. Given any Borel subset $U \subset \mathcal{G}^{(0)}$ of the unit space, there is no general way to endow $\mathcal{G}|_U$ with the structure of measured groupoid unless $U$ has positive $\nu$-measure. In this case, it is sufficient to consider the renormalized restriction $\nu|_U$ of $\nu$ to $U$ and the restriction $\lambda|_U$ of $\lambda$ to $\mathcal{G}|^{(1)}_U$. When $U$ is a conull subset, we call the restricted measured groupoid $(\mathcal{G}|_U,\lambda|_U,\nu|_U)$ \emph{inessential contraction}. 
\end{es}

\begin{defn}\label{def:negligible}
 A subset $U\subset \mathcal{G}^{(0)}$ is \emph{negligible} if $\nu(\mathcal{G}U)=0$, namely its saturation is a null set.  
\end{defn}

An algebraic homomorphism from $\mathcal{G}$ to $\mathcal{H}$ is a functor $\varphi:\mathcal{G}\to \mathcal{H}$. Thus, it consists of a pair of functions, one between objects and one between morphisms. We denote by $\varphi_1:\mathcal{G}^{(1)}\to \mathcal{H}^{(1)}$ the map at the level of morphisms and by $\varphi_0:\mathcal{G}^{(0)}\to \mathcal{H}^{(0)}$ the map at the level of objects. Since $\mathcal{G}^{(0)} \subset \mathcal{G}^{(1)}$, an algebraic morphism $\varphi:\mathcal{G} \rightarrow \mathcal{H}$ is completely determined by $\varphi_1$. Thus, we can identify $\varphi$ with the map $\varphi_1$.

\begin{defn}\label{def:homo}
Let $(\mathcal{G},\lambda,\nu)$ and $(\mathcal{H},\lambda',\nu')$ be two measured groupoids. An algebraic morphism $\varphi:\mathcal{G} \rightarrow \mathcal{H}$ is a \emph{strict homomorphism} if both $\varphi_1$ and $\varphi_0$ are Borel and it holds that $\nu(\varphi_0^{-1}(U))=0$ for every negligible $U\subset \mathcal{H}^{(0)}$. 

  A Borel map $\varphi_1:\mathcal{G}^{(1)} \to \mathcal{H}^{(1)}$ is a \emph{homomorphism} if the restriction of $\varphi_1$ to some inessential contraction $\mathcal{G}|_V^{(1)}$ is a strict homomorphism. 

  A \emph{strict similarity} between strict homomorphisms $\varphi,\psi: \mathcal{G}\to \mathcal{H}$ is a Borel map $h: \mathcal{G}^{(0)}\to \mathcal{H}^{(1)}$ such that $$\psi(g)=h(t(g))\varphi(g)h(s(g))^{-1}$$
  makes sense and it holds for every $g\in \mathcal{G}$. In this case we say that $\varphi$ and $\psi$ are \emph{strictly similar}, writing $\varphi\approx \psi$. 

  Two homomorphisms $\varphi_1,\psi_1: \mathcal{G}^{(1)} \to \mathcal{H}^{(1)}$ are \emph{similar} if the restrictions of both $\varphi_1$ and $\psi_1$ to some inessential contraction $\mathcal{G}|_V^{(1)}$ are strictly similar. In this case we write $\varphi_1 \sim \psi_1$. 

  Two measured groupoids $(\mathcal{G},\lambda,\nu)$ and $(\mathcal{H},\lambda',\nu')$ are \emph{similar} if there exist homomorphisms
  $\varphi_1: \mathcal{G}^{(1)} \rightarrow \mathcal{H}^{(1)}$ and $\psi_1: \mathcal{H}^{(1)} \rightarrow \mathcal{G}^{(1)}$ such that $\psi_1 \circ \varphi_1 \sim \id_{\mathcal{G}^{(1)}}$ and $\varphi_1\circ \psi_1\sim \id_{\mathcal{H}^{(1)}}$. Both compositions make sense since strict homomorphisms are composable by \cite[Corollary 6.3]{ramsay}. 
\end{defn}
 
\begin{nota}
We will refer to a morphism $\varphi_1:\mathcal{G}^{(1)} \rightarrow \mathcal{H}^{(1)}$ via the abuse of notation $\varphi:\mathcal{G} \rightarrow \mathcal{H}$ and we will say that $\varphi$ \emph{is determined by} $\varphi_1$. 
\end{nota}

\begin{es}[Morphisms where the codomain is a group]\label{es:target:group}
Let $(H,\tau)$ be a locally compact group with its Haar measure.
\begin{itemize}
\item [(i)] Given a measured groupoid $(\mathcal{G},\lambda,\nu)$, if we consider a (strict) morphism of measured groupoids $\rho:\mathcal{G} \rightarrow H$, then the only negligible set of the units of $H$ is the empty set, thus the condition of regularity on units introduced in Definition \ref{def:homo} is actually empty. As a consequence any measurable (strict) algebraic morphism $\rho:\mathcal{G} \rightarrow H$ is a (strict) morphism of measured groupoids in the sense of Definition \ref{def:homo}.  From now on, when the codomain is a locally compact group, the word (strict) homomorphism will simply refer to a measurable algebraic (strict) homomorphism.

\item [(ii)] Let $G$ be a locally compact, second countable, unimodular group and let $(X,\mu)$ a standard Borel probability $G$-space. If we consider a morphism $\sigma:G \ltimes X \rightarrow H$ of the action groupoid, by the previous item, $\sigma$ is simply a measurable function which satisfies 
$$
\sigma(g_1,g_2x) \cdot \sigma(g_2,x)=\sigma(g_1g_2,x),
$$
for all $g_1,g_2 \in G$ and almost $\mu$-every $x \in X$. This means that $\sigma$ is precisely a \emph{measurable cocycle} \cite[Definition 4.2.1]{zimmer:libro}. In the same setting two morphisms $\sigma_1,\sigma_2:G \ltimes X \rightarrow H$ are similar if and only if they are \emph{cohomologous} when viewed as measurable cocycle \cite[Definition 4.2.2]{zimmer:libro}. 
\end{itemize}
\end{es}

\begin{oss}
The definitions of morphisms and similarities adopted here first appeared in the work of Ramsay \cite[Definition 6.1]{ramsay}. We caution the reader about the lack of uniformity regarding these notions in the literature. For instance, in a recent work by the authors \cite{sarti:savini:23}, morphisms and similarities were defined differently. For the purposes of this note, we prefer to adopt Ramsay's point of view, since we think it is the correct one here.
\end{oss}

Given an ergodic groupoid $(\mathcal{G},\lambda,\nu)$ and a morphism $\rho:\mathcal{G} \rightarrow H$ into a locally compact, second countable group, recall that the \emph{essential image} of $\rho$ is the support of the measure $\rho_*(\nu\circ\lambda)$ in $H$. The following result allows to control the essential image of $\rho$ by looking to a suitable inessential contraction of $\mathcal{G}$. 

\begin{lemma}\label{lemma:inessential:contraction:hull}
Consider an ergodic groupoid $(\mathcal{G},\lambda,\nu)$ and let $H$ be a locally compact, second countable group. Given a homomorphism $\rho:\mathcal{G} \rightarrow H$, suppose that there exists a subgroup $L<H$ such that $\rho(g) \in L$, for almost every $g \in \mathcal{G}^{(1)}$. Then there exists a conull subset $U \subset \mathcal{G}^{(0)}$ such that 
$$
\rho\left(\mathcal{G}|_U^{(1)}\right) \subseteq L. 
$$
\end{lemma}
\begin{proof}
By assumption, the set
$$
\Sigma:=\{ g \in \mathcal{G}^{(1)}  \ | \ \rho(g) \in L\} 
$$
has full measure. Up to intersecting $\Sigma$ with 
$$
\Sigma^{-1}:=\{ g \in \mathcal{G}^{(1)}  \ | \ g^{-1} \in \Sigma \},
$$
we can suppose that $\Sigma$ is a symmetric subset of full measure. To obtain the existence of an inessential contraction contained in $\Sigma$, by \cite[Lemma 5.2]{ramsay} it is sufficient to show that $\Sigma$ is a \emph{multiplicative} set, namely it is closed under multiplication. The latter fact is obvious, being $\Sigma$ the preimage of the subgroup $L$. This concludes the proof. 
\end{proof}

\subsection{Algebraic representability of ergodic groupoids} \label{sec:alg:rep}

We briefly recall the definition of algebraic representations for ergodic groupoids introduced in \cite{sarti:savini:boundaries}.

Let $\kappa$ be a field with a non-trivial absolute value so that the induced topology is complete and separable. Let $(\mathcal{G},\lambda,\nu)$ be an ergodic groupoid. Consider a homomorphism $\rho:\mathcal{G} \rightarrow H$ into the $\kappa$-points $H=\mathbf{H}(\kappa)$ of an algebraic $\kappa$-group $\mathbf{H}$. An \emph{algebraic representation of $(\mathcal{G},\lambda,\nu)$ relative to $\rho$} is the datum of a pair $(\mathbf{V},\varphi)$, where $\mathbf{V}$ is an algebraic $\kappa$-variety with a $\kappa$-algebraic $\mathbf{H}$-action, and 
$$
\varphi:\mathcal{G}^{(0)} \rightarrow \mathbf{V}(\kappa)
$$
is a $\rho$-equivariant measurable map, namely 
\begin{equation}\label{eq:rho:equivariance}
\varphi(t(g))=\rho(g)\varphi(s(g)),
\end{equation}
for almost $(\nu \circ \lambda)$-every $g \in \mathcal{G}^{(1)}$. Given a homomorphism $\rho:\mathcal{G} \rightarrow H$, we will refer to any algebraic representation of $(\mathcal{G},\lambda,\nu)$ relative to $\rho$ via the pair $(\mathbf{V},\varphi_{\mathbf{V}})$, where $\mathbf{V}(\kappa)$ is the codomain of $\varphi_{\mathbf{V}}$. A \emph{morphism} between two algebraic representations $(\mathbf{V},\varphi_{\mathbf{V}})$ and $(\mathbf{W},\varphi_{\mathbf{W}})$ is a $\mathbf{H}$-equivariant $\kappa$-algebraic map $\Psi:\mathbf{V} \rightarrow \mathbf{W}$ such that 
\begin{equation}\label{eq:morphism:rep}
\varphi_{\mathbf{W}}(x)=(\Psi \circ \varphi_{\mathbf{V}})(x),
\end{equation}
for almost $\nu$-every $x \in \mathcal{G}^{(0)}$. We will identify two algebraic representations and the morphisms between them when they coincide on a subset of full measure. We will denote the category of (classes of) algebraic representations of $(\mathcal{G},\lambda,\nu)$ relative to $\rho$ by $\mathrm{Rep}_\kappa(\rho)$. 

An algebraic representation $(\mathbf{V},\varphi_{\mathbf{V}})$ is of \emph{coset type} if $\mathbf{V}=\mathbf{H}/\mathbf{L}$, for some $\kappa$-algebraic subgroup $\mathbf{L}<\mathbf{H}$, where the $\mathbf{H}$-action is the standard one. By \cite[Theorem 5.5]{sarti:savini:boundaries} the category $\mathrm{Rep}_\kappa(\rho)$ has an initial object of coset type. Let $(\mathbf{H}/\mathbf{L},\varphi)$ be such an initial object. We call the pair $(\mathbf{L},\varphi)$ the \emph{algebraic gate} associated to $\rho$ and $\mathbf{L}$ is the \emph{algebraic hull} of $\rho$. The representation $\rho$ is \emph{Zariski dense} if its algebraic hull coincides with $\mathbf{H}$. 

Let $U \subset \mathcal{G}^{(0)}$ be a full $\nu$-measure subset of $\mathcal{G}^{(0)}$. We can restrict $\rho$ to the inessential contraction $\mathcal{G}|_U$ to obtain another morphism $\rho|_{\mathcal{G}|_U}$ with values in $H$. As a consequence, we are allowed to consider the category of algebraic representations $\mathrm{Rep}_\kappa(\rho|)$ associated to such restriction. In this context, we have a functor
$$
\mathfrak{R}_U:\mathrm{Rep}_\kappa(\rho) \rightarrow \mathrm{Rep}_\kappa(\rho|)
$$
defined as follows: if $(\mathbf{V},\varphi)$ is an algebraic representation of $\rho$, then $\mathfrak{R}_U((\mathbf{V},\varphi))=(\mathbf{V},\varphi|_U)$ is obtained by restricting $\varphi$ to $U$; given a morphism $\Psi:\mathbf{V} \rightarrow \mathbf{W}$ of two objects $(\mathbf{V},\varphi),(\mathbf{W},\psi)$ in  $\mathrm{Rep}_\kappa(\rho)$, we define $\mathfrak{R}_U(\Psi)=\Psi$. 

\begin{lemma}\label{lemma:eq:categories}
Let $\kappa$ be a local field of characteristic zero. Consider an ergodic groupoid $(\mathcal{G},\lambda,\nu)$ and let $\rho:\mathcal{G} \rightarrow H$ be a morphism into the $\kappa$-points $H=\mathbf{H}(\kappa)$ of an algebraic $\kappa$-group. For any full $\nu$-measure subset of $\mathcal{G}^{(0)}$, there functor $\mathfrak{R}_U$ is a well-defined equivalence of categories. As a consequence, the algebraic hull of an algebraic representation of $(\mathcal{G},\lambda,\nu)$ relative to $\rho$ is unchanged when we restrict $\rho$ to any inessential contraction.
\end{lemma}

\begin{proof}
The functor $\mathfrak{R}_U$ is well-defined by the fact that $U$ has full measure in $\mathcal{G}^{(0)}$ and $\mathcal{G}|_U^{(1)}$ has full measure in $\mathcal{G}^{(1)}$. 

To prove that $\mathfrak{R}_U$ is an equivalence of categories, we need to prove that it is fully faithful and essentially surjective. We fix two algebraic representations $(\mathbf{V},\varphi)$ and $(\mathbf{W},\psi)$ in $\mathrm{Rep}_\kappa(\rho)$. We start showing that $\mathfrak{R}_U$ is faithful. Given two morphisms $\Psi_1,\Psi_2$ between $(\mathbf{V},\varphi)$ and $(\mathbf{W},\psi)$, the condition $\mathfrak{R}_U(\Psi_1)=\mathfrak{R}_U(\Psi_2)$ readily implies that $\Psi_1=\Psi_2$ for almost every $x \in U$, and hence for almost every $x \in \mathcal{G}^{(0)}$.

We now prove that $\mathfrak{R}_U$ is full. Let $\Psi$ be a morphism between $(\mathbf{V},\varphi|_U)$ and $(\mathbf{W},\psi|_U)$. Since Equation \eqref{eq:morphism:rep} holds for almost every $x \in U$ and $U$ has full measure in $\mathcal{G}^{(0)}$, we can legitimately define $\Psi$ as a morphism between $(\mathbf{V},\varphi)$ and $(\mathbf{W},\psi)$. 

We are left to show that $\mathfrak{R}_U$ is essentially surjective. Take an algebraic representation $(\mathbf{V},\varphi)$ in $\mathrm{Rep}_\kappa(\rho|)$. We need to promote it to an algebraic representation of $\rho$. In order to do that, we define $\widehat{\varphi}:\mathcal{G}^{(0} \rightarrow \mathbf{V}(\kappa)$ to be $\varphi$ on $U$ and to be constantly equal to an arbitrary fixed point of $\mathbf{V}(\kappa)$ on $\mathcal{G}^{(0)} \setminus U$. The map $\widehat{\varphi}$ is measurable and it still satisfies Equation \eqref{eq:rho:equivariance} since $\mathcal{G}|^{(1)}_U$ has full measure in $\mathcal{G}^{(1)}$. This concludes the proof of the equivalence. The statement about the algebraic hull follows by the definition of $\mathfrak{R}_U$.
\end{proof}

\begin{oss}\label{immagine:kpunti:omogenei}
Let $\kappa$ be a local field of characteristic zero. Given an ergodic groupoid $(\mathcal{G},\lambda,\nu)$ and the $\kappa$-points $H=\mathbf{H}(\kappa)$ of an algebraic $\kappa$-group, we consider a homomorphism $\rho:\mathcal{G}  \rightarrow H$. For any algebraic representation of $\rho$ of coset type $(\mathbf{H}/\mathbf{L},\varphi)$, we claim that the measurable map 
$$
\varphi:\mathcal{G}^{(0)} \rightarrow (\mathbf{H}/\mathbf{L})(\kappa),
$$
has essential image contained in $H/L$, where $L=\mathbf{L}(\kappa)$. By the ergodicity of the groupoid $(\mathcal{G},\lambda,\nu)$ and by the fact that the $H$-orbits on $(\mathbf{H}/\mathbf{L})(\kappa)$ are locally closed, the essential image of $\varphi$ is contained in a single $H$-orbit, say $Hx$ for $x \in (\mathbf{H}/\mathbf{L})(\kappa)$. By \cite[Proposition 2.1]{bader:furman:compositio} the $\kappa$-algebraic map $\mathbf{H}/\mathbf{L} \rightarrow \mathbf{H}x$ restricts to a homeomorphims $H/L \rightarrow Hx$. Hence, up to composing with the inverse of the previous homeomorphism, we can suppose that the codomain of $\varphi$ is given by $H/L$.
\end{oss}

\section{Transverse systems and similarity}\label{section:transverse}
Let $G$ be a locally compact, second countable, unimodular group. Consider an \emph{ergodic transverse $G$-system} $(X,\mu,Y)$, namely an ergodic probability measure-preserving $G$-action on $(X,\mu)$ and a Borel cross section $Y\subset X$ (that is a Borel subset which satisfies $GY=X$) such that 
$$Y_x\coloneqq\{g\in G\,,\, gx\in Y\}$$ is locally finite for every $x\in X$. In this context, \emph{locally finite} means that the cardinality $|Y_x \cap K|$ is finite, for every compact subset $K \subset G$. If additionally we have that
$$
\int_X |Y_x \cap K|d\mu(x) <\infty,
$$
we say that the ergodic transverse $G$-system is \emph{integrable}. 

Given an ergodic transverse $G$-system $(X,\mu,Y)$, there is a canonical way to define a measure $\nu$ on $Y$.

\begin{prop}[{Refined Campbell theorem, \cite[Proposition\ 4.2]{ABC}}]\label{Campbell} 
Let $m_G$ the Haar measure on $G$. There exists a unique $\sigma$-finite measure $\nu$ on $Y$ such that for every Borel function $w: G \to [0, \infty]$ with $m_G(w) = 1$ and every Borel function $f: Y \to [0, \infty]$ we have
\begin{equation}
\int_Y f(y) \, d \nu(y) = \int_X \sum_{g \in Y_x} f(g x) w(g) \, \, d \mu(x).
\end{equation}
Moreover, for every Borel function $F: G \times Y \to [0, \infty]$ we have the Campbell formula
\begin{equation}\label{equation:formula:integral}
\int_G \int_Y F(g,y) \, dm_G(g)d\nu(y) = \int_X \sum_{g \in Y_x} F(g^{-1}, g x)\, d\mu(x).
\end{equation}
\end{prop}

If we require that the ergodic transverse $G$-system $(X,\mu,Y)$ is also integrable, the measure $\nu$ must be finite by \cite[Proposition 5.1]{ABC} and hence it can be renormalized to a probability measure. For this reason, we are going to focus our attention only on ergodic integrable transverse $G$-systems. We fix such a system $(X,\mu,Y)$. Since $Y_x$ is locally finite, the restriction groupoid $\mathcal{G}\coloneqq G\ltimes X|_Y$ has discrete $t$-fibers.
We equip $\mathcal{G}$ with the transverse measure $\nu$ and the Haar system $\lambda$ given by the counting measures on the fibers and we call it the \emph{transverse measured groupoid} associated to the ergodic integrable transverse $G$-system $(X,\mu,Y)$.

\begin{es}\label{example:lattice}
If we fix a lattice $\Gamma<G$, then we can consider the usual action of $G$ on the quotient space $G/\Gamma$, where the latter is endowed with its unique $G$-invariant probability measure $\mu$. By setting $X=G/\Gamma$ and $Y=\{e\Gamma\}$, we have that $(X,\mu,Y)$ is an ergodic integrable transverse $G$-system such that the associated transverse measured groupoid is isomorphic to $\Gamma$ equipped with the Dirac mass multiplied by $\mathrm{covol}(\Gamma)^{-1}$.  
\end{es}

Denote by $a: G\times Y\to X$ the multiplication map $a(g,y)=gy$ and choose a Borel map $\beta:X\to G$ such that $$x\mapsto (\beta(x)^{-1},\beta(x)x)$$
is a Borel section of $a$ and $\beta(y)=e$ for all $y \in Y$. The existence of such $\beta$ is ensured by the fact that $a$ is countable-to-one \cite[Theorem 18.10]{kechris}.
Define the measurable cocycle
$$\sigma:G\times X\to G\,,\quad \sigma(g,x)\coloneqq \beta(gx)g\beta(x)^{-1}\,.$$
Hence $\sigma(g,y)=g$ if and only if $(g,y)\in \mathcal{G}$.

%

From now on, given an ergodic integrable transverse $G$-system $(X,\mu,Y)$, we consider $Y$ endowed with the measure $\nu$ of Proposition \ref{Campbell}. 
The following is the analogous of \cite[Theorem 6.17]{ramsay}. 
\begin{rec_thm}[\ref{thm:restriction:similar}]
  Let $G$ be a locally compact, second countable, unimodular group and $(X,\mu,Y)$ be an ergodic integrable transverse $G$-system. Let $\nu$ be the transverse measure on $Y$. Then $(G\ltimes X,m_G,\mu)$ and $(\mathcal{G},\nu)$ are similar ergodic groupoids. 
\end{rec_thm}
\begin{proof}
  The ergodicity of $(\mathcal{G},\nu)$ follows by \cite[Proposition 5.5]{ABC}.
  Retain the notations above and
  consider the inclusion $\iota:\mathcal{G} \to G\ltimes X$ and the map 
  $\psi: G\ltimes X\to \mathcal{G}$ defined by 
  $$\psi (g,x)\coloneqq (\sigma(g,x),\beta(x)x)\,.$$
  We claim that they determine two homomorphisms. The fact that they are algebraic morphisms is an easy computation. To check the behaviour on negligible sets, we consider a Borel subset $U\subset X$. Then by \cite[Proposition 5.3]{ABC} we have
 \begin{equation}\label{eq:null:sets}
   \nu(GU\cap Y)=0 \iff \mu(GU)=0\,.
 \end{equation}
  Hence, if $U$ is negligible, then $\mu(GU)=0$ and $$\nu(\iota_0^{-1}(U))=\nu(U\cap Y)\leq \nu(GU\cap Y)=0\,,$$
where $\iota_0$ is the map induced by $\iota$ on units. 
  
On the other hand, if $V\subset Y$ is negligible, then $\nu(\mathcal{G}V)=0$. Moreover, since $\mathcal{G}V=GV\cap Y$, by Equation \eqref{eq:null:sets} we have that $\mu(GV)=0$. Hence
  $$\mu(\psi_0^{-1}(V))\leq \mu(GV)=0\,,$$
where $\psi_0$ is the map induced by $\psi$ on units, namely $x \mapsto \beta(x)x$.
  This finishes the proof of the claim. 
  We conclude by observing that $\psi\circ\iota=\id_{\mathcal{G}}$ and that $\iota\circ \psi$ is similar to $\id_{G\ltimes X}$ via 
  $$h:X\to G\ltimes X\,,\quad h(x)\coloneqq (\beta(x),x)\,.$$
\end{proof}

\begin{oss}
Retain the setting of Theorem \ref{thm:restriction:similar}. In \cite{ramsay} the author defines a measure $\theta$ on $Y$ turning $\mathcal{G}$ into a measured groupoid such that, when $(G\ltimes X,m_G,\mu)$ is ergodic, then so is $(\mathcal{G},\theta)$. With such a measure, Ramsay proves that $(G\ltimes X,m_G,\mu)$ and $(\mathcal{G},\theta)$ are similar in the sense of Definition \ref{def:homo}. One could wonder what is the relation between Ramsay's measure $\theta$ and the measure $\nu$. We include here a brief discussion.

 Ramsay's quasi-invariant measure on $Y$ is constructed as follows: consider a symmetric probability measure $\varrho$ on $G\ltimes  X$ in the measure class of $m_G \otimes \mu$ and consider the pushforward of $\varrho$ via the morphism
 $$\psi:G\times X\to \mathcal{G}\,,\quad (g,x)\mapsto (\sigma(g,x),\beta(x)x)\,.$$
If we define $\theta$ to be the image of $\psi_*(\varrho)$ via the target map $t:\mathcal{G}^{(1)}\to Y$, then $(\mathcal{G},\theta)$ is in fact a measured groupoid. That is, $\theta$ is quasi-invariant with respect to $\mathcal{G}$ equipped with the Haar system determined by the counting measures. 

By the uniqueness statement in Proposition \ref{Campbell}, the measure $\theta$ does not coincide with $\nu$ in general. Nevertheless, we claim that $\theta$ and $\nu$ have the same negligible sets.  To see this, let $U\subset Y$ such that $\nu(\mathcal{G}U)=0$. Then, since $\mathcal{G}U=GU\cap Y$, by \cite[Proposition 5.3]{ABC} we have $\mu(GU)=0$. Since $\mu$ and $t_\ast \varrho$ are equivalent, then $U$ is negligible for $\theta$. The converse is analogous. 

 As a consequence, $\varphi:\mathcal{G}\to \mathcal{H}$ is a (strict) homomorphism for $(\mathcal{G},\nu)$ if and only if it is a (strict) homomorphism for $(\mathcal{G},\theta)$.
\end{oss}

The similarity between $G\ltimes X$ and $\mathcal{G}$ is the only ingredient that we needed to prove Theorem \ref{thm:property:T}.
\begin{proof}[Proof of Theorem \ref{thm:property:T}]
We know that $G \ltimes X$ and $\mathcal{G}$ are similar by Theorem \ref{thm:restriction:similar}. Let $\psi:G \ltimes X \rightarrow \mathcal {G}$ and $\iota:\mathcal{G} \rightarrow G \ltimes X$ be the morphisms realizing the similarity. By composing with $\psi$ we obtain a morphism
$$
\rho \circ \psi:G \ltimes X \rightarrow H,
$$
where $H$ is amenable. By \cite[Theorem 9.1.1]{zimmer:libro} we have that $\rho \circ \psi$ is similar to a morphism $\rho': G \ltimes X \rightarrow H$ whose image is contained into a compact group. Then $\rho' \circ \iota$ is a morphism similar to $\rho$, whose image is contained in a compact group. 
\end{proof}
  
\section{Zariski dense representations}\label{zariski:dense:rep}

In this section we will focus our attention on Zariski dense morphisms of an ergodic groupoid $(\mathcal{G},\lambda,\mu)$. We are going to prove that Zariski density is invariant for similar morphisms. Additionally, similar ergodic groupoids have the same set of similarity classes of Zariski dense morphisms. 

We begin with the following characterization of Zariski density.
\begin{prop}\label{algebraic:hull:zariski:dense}
Let $\kappa$ be a local field of characteristic zero. Consider an ergodic groupoid $(\mathcal{G},\lambda,\nu)$ and let $H=\mathbf{H}(\kappa)$ be the $\kappa$-points of an algebraic $\kappa$-group $\mathbf{H}$.  Fix a homomorphism $\rho:\mathcal{G}  \rightarrow H$. Then $\rho$ is Zariski dense if and only if the set
$$
E(\rho):=\{\mathbf{L} < \mathbf{H} \ | \ \exists \widehat{\rho} \sim \rho \ \text{such that} \ \mathrm{Ess\,Im}(\widehat{\rho}) \subset L=\mathbf{L}(\kappa) \}
$$
is empty. Here $\mathrm{Ess\,Im}(\widehat{\rho})$ refers to the essential image of $\widehat{\rho}$ and $\mathbf{L}$ is a proper $\kappa$-subgroup.
\end{prop}

\begin{proof}
We start assuming that $\rho$ is Zariski dense. By contradiction, suppose that the set $E(\rho)$ is not empty. As a consequence, there exists a proper $\kappa$-algebraic subgroup $\mathbf{L}<\mathbf{H}$ so that $\rho$ is similar to a homomorphism $\widehat{\rho}:\mathcal{G} \rightarrow H$ with essential image in $L=\mathbf{L}(\kappa)$. By Lemma \ref{lemma:inessential:contraction:hull}, we can suppose that $\widehat{\rho}$ sends an inessential contraction of $\mathcal{G}$ entirely in $L$. Thus, up to passing to an inessential contraction, we suppose that there exists a measurable map 
$$
\theta:\mathcal{G}^{(0)} \rightarrow H,
$$
such that 
$$
\theta(t(g))\widehat{\rho}(g)=\rho(g)\theta(s(g))
$$
and $\widehat{\rho}(g) \in L$, for all $g \in \mathcal{G}^{(1)} $. By composing $\theta$ with the quotient projection $H\rightarrow H/L$ and the inclusion $H/L \rightarrow (\mathbf{H}/\mathbf{L})(\kappa)$ we obtain a measurable map 
$$
\overline{\theta}:\mathcal{G}^{(0)} \rightarrow (\mathbf{H}/\mathbf{L})(\kappa)
$$
which is $\rho$-equivariant. This contradicts the Zariski density of $\rho$. 

Assume now that the set $E(\rho)$ is empty. By contradiction, suppose that $\rho$ is not Zariski dense. Then there exists a proper $\kappa$-algebraic subgroup $\mathbf{M}<\mathbf{H}$ and a $\rho$-equivariant measurable map 
$$
\varphi:\mathcal{G}^{(0)} \rightarrow (\mathbf{H}/\mathbf{M})(\kappa).
$$
Thanks to Remark \ref{immagine:kpunti:omogenei}, we can suppose that $\varphi$ takes values in $H/M$, namely
$$
\varphi:\mathcal{G}^{(0)} \rightarrow H/M,
$$
where $M=\mathbf{M}(\kappa)$. By \cite[Corollary A.8]{zimmer:libro}, we can compose $\varphi$ with a measurable section $s:H/M \rightarrow H$ to obtain a measurable map 
$$
\theta:\mathcal{G}^{(0)} \rightarrow H.
$$
If we define 
$$
\widehat{\rho}:\mathcal{G}  \rightarrow H, \ \ \ \widehat{\rho}(g):=\theta(t(g))^{-1}\rho(g)\theta(s(g)),
$$
then $\theta$ is a similarity between $\rho$ and $\widehat{\rho}$. Moreover $\widehat{\rho}$ has essential image contained in $M$. This contradicts our initial assumption and concludes the proof. 
\end{proof}

The previous result shows that, for Zariski dense homomorphisms, one cannot find in the same similarity class a homomorphism whose essential image is contained in a proper subgroup. A direct consequence of Proposition \ref{algebraic:hull:zariski:dense} is the following:

\begin{cor}\label{cor_zariski_dense_invariant}
Let $\kappa$ be a local field of characteristic zero. Let $(\mathcal{G},\lambda,\nu)$ be an ergodic groupoid and let $H=\mathbf{H}(\kappa)$ be the $\kappa$-points of an algebraic $\kappa$-group $\mathbf{H}$. Consider two homomorphisms
$$
\rho_1,\rho_2:\mathcal{G} \rightarrow H. 
$$
Suppose that $\rho_1$ and $\rho_2$ are similar. Then $\rho_1$ is Zariski dense if and only if $\rho_2$ is Zariski dense. 
\end{cor}

\begin{proof}
By contradiction suppose that $\rho_1$ is Zariski dense and $\rho_2$ is not. By Proposition \ref{algebraic:hull:zariski:dense} the set $E(\rho_2)$ is not empty, hence there exist a proper $\kappa$-algebraic subgroup $\mathbf{L}<\mathbf{H}$ and a homomorphism $\rho_2':\mathcal{G} \rightarrow H$ whose essential image is contained in $L=\mathbf{L}(\kappa)$. Since similarity is an equivalence relation, this implies that $E(\rho_1)$ is not empty, leading to a contradiction.
\end{proof}

The previous result is a particular instance of a more general fact. Indeed, we are going to prove that the conjugacy class of the algebraic hull is an invariant for similarity.

\begin{prop}\label{prop_algebraic_hull_invariant}
Let $\kappa$ be a local field of characteristic zero. Let $(\mathcal{G},\lambda,\nu)$ be an ergodic groupoid and let $H=\mathbf{H}(\kappa)$ be the $\kappa$-points of an algebraic $\kappa$-group $\mathbf{H}$. Consider two homomorphisms
$$
\rho_1,\rho_2:\mathcal{G} \rightarrow H. 
$$
Let $\mathbf{L}_i < \mathbf{H}$ be the algebraic hull of $\rho_i$, for $i=1,2$. If $\rho_1$ and $\rho_2$ are similar, then $\mathbf{L}_1$ and $\mathbf{L}_2$ are conjugate. 
\end{prop}

\begin{proof}
By definition, we know that there exist conull subsets $U_1,U_2 \subset \mathcal{G}^{(0)}$ such that 
$$
\rho_1:\mathcal{G}|_{U_1} \rightarrow H, \ \ \ \ \rho_2:\mathcal{G}|_{U_2}  \rightarrow H, 
$$
are strict homomorphisms. Since $\rho_1$ and $\rho_2$ are similar, there must exist a conull subset $U \subset U_1 \cap U_2$ and measurable map 
$$
\theta:\mathcal{G}^{(0)} \rightarrow H, 
$$
such that 
$$
\rho_2(g)=\theta(t(g))\rho_1(g)\theta(s(g))^{-1},
$$
for all $g \in \mathcal{G}|_U^{(1)} $. 

By Lemma \ref{lemma:eq:categories} restricting to an inessential contraction does not modify the algebraic hull. Thus, without loss of generality, we can suppose that $\rho_1$ and $\rho_2$ are strictly similar strict morphisms with algebraic hull $\mathbf{L}_1$ and $\mathbf{L}_2$, respectively. 

If $\mathbf{L}_i$ is the algebraic hull of $\rho_i$, for $i=1,2$, there must exist a measurable map 
$$
\varphi_i:\mathcal{G}^{(0)} \rightarrow  (\mathbf{H}/\mathbf{L_i})(\kappa),
$$
such that 
$$
\varphi_i(t(g))=\rho_i(g)\varphi(s(g)),
$$
for almost every $g \in \mathcal{G}^{(1)} $ and for $i=1,2$. Additionally $(\mathbf{H}/\mathbf{L}_i,\varphi_i)$ is the initial object in the category of algebraic representations, again for $i=1,2$. 

We define
$$
\widehat{\varphi_1}:\mathcal{G}^{(0)} \rightarrow (\mathbf{H}/\mathbf{L}_1)(\kappa), \ \ \ \widehat{\varphi_1}(x):=\theta(x)\varphi_1(x).
$$
We verify that $\widehat{\varphi_1}$ is $\rho_2$-equivariant, namely:
\begin{align*}
\widehat{\varphi_1}(t(g))&=\theta(t(g))\varphi_1(t(g))\\
&=\theta(t(g))\rho_1(g)\varphi_1(s(g))\\
&=\rho_2(g)\theta(s(g))\varphi_1(s(g))=\rho_2(g)\widehat{\varphi_1}(s(g)). 
\end{align*}
Since $(\mathbf{H}/\mathbf{L}_2,\varphi_2)$ is an initial object in the category of algebraic representations of $\rho_2$, there must exist a $\mathbf{H}$-equivariant $\kappa$-map 
$$
F:\mathbf{H}/\mathbf{L}_2 \rightarrow \mathbf{H}/\mathbf{L}_1
$$
such that $\widehat{\varphi_1}=F \circ \varphi_2$. 

In a similar way, by setting 
$$
\widehat{\varphi_2}:\mathcal{G}^{(0)} \rightarrow (\mathbf{H}/\mathbf{L}_2)(\kappa), \ \ \ \widehat{\varphi_2}(x):=\theta(x)^{-1}\varphi_2(x),
$$
we obtain a $\rho_1$-equivariant map. Since $(\mathbf{H}/\mathbf{L}_1,\varphi_1)$ is an initial object in the category of algebraic representation of $\rho_1$, there exists a $\mathbf{H}$-equivariant $\kappa$-map 
$$
G:\mathbf{H}/\mathbf{L}_1 \rightarrow \mathbf{H}/\mathbf{L}_2
$$
such that $\widehat{\varphi_2}=G \circ \varphi_1$. The existence of both the maps $F$ and $G$, implies that $\mathbf{L}_1$ and $\mathbf{L}_2$ are conjugate. This concludes the proof. 
\end{proof}

The fact that Zariski density is invariant under similarity by Corollary \ref{cor_zariski_dense_invariant} allows to consider the set of classes of Zariski dense homomorphisms from a given ergodic measured groupoid into an algebraic group. 
\begin{defn}\label{def:zariski:dense:classes}
Consider a local field $\kappa$ of characteristic zero. Let $(\mathcal{G},\lambda,\nu)$ be an ergodic groupoid and let $H=\mathbf{H}(\kappa)$ be the $\kappa$-points of an algebraic group $\mathbf{H}$. We denote by 
$$
\mathrm{Hom}_{ZD}[\mathcal{G},H]:=\{ [\rho] \ | \ \rho:\mathcal{G} \rightarrow H \ \textup{is Zariski dense} \}
$$
the set of similarity classes of Zariski dense homomorphisms. 
\end{defn}

Given similar ergodic groupoids, our next goal is to compare their classes of Zariski dense homomorphisms into a given algebraic group. In fact, any similarity descends to a bijection on $\mathrm{Hom}_{ZD}$.
\begin{prop}\label{prop:zariski:dense:similarity:invariant}
Let $(\mathcal{G}_1,\lambda_1,\nu_1)$ and $(\mathcal{G}_2,\lambda_2,\nu_2)$ be two ergodic groupoids. Given a local field $\kappa$ of characteristic zero, consider $H=\mathbf{H}(\kappa)$ the $\kappa$-points of an algebraic $\kappa$-group $\mathbf{H}$. If $\psi:\mathcal{G}_1 \rightarrow \mathcal{G}_2$ is a similarity, then the bijection
$$
\Psi^*:\mathrm{Hom}[\mathcal{G}_2,H] \rightarrow \mathrm{Hom}[\mathcal{G}_1,H]
$$
restricts to a well-defined bijection on Zariski dense classes, namely
$$
\Psi^*:\mathrm{Hom}_{ZD}[\mathcal{G}_2,H] \rightarrow \mathrm{Hom}_{ZD}[\mathcal{G}_1,H]. 
$$
\end{prop}

\begin{proof}
Consider a Zariski dense homomorphism $\rho:\mathcal{G}_2 \rightarrow H$ and take the composition $\rho \circ \psi:\mathcal{G}_1 \rightarrow H$. By contradiction, assume that the latter is not Zariski dense.  By Proposition \ref{algebraic:hull:zariski:dense} there must exist a proper $\kappa$-algebraic subgroup $\mathbf{L}<\mathbf{H}$ and a morphism $\alpha:\mathcal{G}_1 \rightarrow H$ similar to $\rho \circ \psi$ whose essential image is contained in $L=\mathbf{L}(\kappa)$. By Lemma \ref{lemma:inessential:contraction:hull} we can suppose that there exists a conull subset $U \subseteq \mathcal{G}_1^{(0)}$ such that
$$
\alpha\left(\mathcal{G}_1|_U^{(1)}\right) \subseteq L. 
$$
Let $\varphi:\mathcal{G}_2 \rightarrow \mathcal{G}_1$ be the homomorphism such that $\varphi \circ \psi$ is similar to $\mathrm{id}_{\mathcal{G}_1}$ and $\psi \circ \varphi$ is similar to $\mathrm{id}_{\mathcal{G}_2}$. Since $U$ is conull, a fortiori its $\mathcal{G}_1$-saturation is conull. By \cite[Lemma 6.6]{ramsay} we can find a homomorphism $\sigma:\mathcal{G}_2 \rightarrow \mathcal{G}_1$ similar to $\varphi$ and such that
$$
\sigma\left(\mathcal{G}_2^{(1)}\right) \subset \mathcal{G}_1|_U^{(1)}. 
$$
As a consequence, the composition $\alpha \circ \sigma$ satisfies
$$
(\alpha \circ \sigma)\left(\mathcal{G}_2^{(1)}\right) \subset L.
$$
Moreover $\alpha \circ \sigma$ is similar to $\rho$ since we have that
$$
\alpha \circ \sigma \sim \alpha \circ \varphi \sim \rho \circ \psi \circ \varphi \sim \rho. 
$$
In this way we find a homomorphism $\alpha \circ \sigma$ similar to $\rho$ whose image is contained in $L$. By Proposition \ref{algebraic:hull:zariski:dense} this contradicts the Zariski density of $\rho$. 

We have shown that 
$$
\Psi^\ast \left(\mathrm{Hom}_{ZD}[\mathcal{G}_2,H] \right) \subset \mathrm{Hom}_{ZD}[\mathcal{G}_1,H].
$$
To show that $\Psi^\ast$ is precisely a bijection on Zariski dense classes, it is sufficient to apply the same argument we followed above to the map $\Phi^\ast$ induced by $\varphi$. This concludes the proof. 
\end{proof}

We are finally ready to prove our superrigidity results. 

\begin{proof}[Proof of Theorem \ref{thm:superrigidity}]
We have a similarity between $G \ltimes X$ and $\mathcal{G}$ by Theorem \ref{thm:restriction:similar}. Let $\psi:G \ltimes X \rightarrow \mathcal{G}$ be a homomorphism realizing such a similarity (the other is $\iota:\mathcal{G} \rightarrow G \ltimes X$). Proposition \ref{prop:zariski:dense:similarity:invariant} guarantees that the map
$$
\Psi:\mathrm{Hom}_{ZD}[\mathcal{G},H] \rightarrow \mathrm{Hom}_{ZD}[G \ltimes X, H]
$$
is a bijection. Then
$$
\rho \circ \psi: G \ltimes X \rightarrow H
$$
is a Zariski dense homomorphism, namely a Zariski dense measurable cocycle by Example \ref{es:target:group}. As a consequence, we are in the right position to apply the superrigidity theorem for cocycles \cite[Theorem 10.1.6]{zimmer:libro}. We are going to distinguish the three cases of \cite[Theorem 10.1.6]{zimmer:libro} separately:

\begin{enumerate}
\item Assume that $\rho \circ \psi$ is similar to a cocycle $\sigma:G \ltimes X \rightarrow H$ with values into a compact subgroup of $H$. Then $\sigma \circ \iota$ is a morphism defined on $\mathcal{G}$ whose image lies into a compact subgroup of $H$ and is similar to $\rho$. 

\item Suppose $\kappa=\mathbb{Q}_{p_i}$, for some $p_i \in S \setminus \{\infty\}$. Then $\rho \circ \psi$ is similar to a cocycle induced by the following composition 
$$
\pi:G \rightarrow G_i \rightarrow H,
$$
where the first map is the projection on the $i$-th factor and the second one is a rational epimorphism. As a consequence, $\rho$ is similar to a representation induced by $\pi$. 

\item For $\kappa=\mathbb{R},\mathbb{C}$ the proof is analogous to the argument of the previous point. 
\end{enumerate}
This concludes the proof. 
\end{proof}

The same strategy applies to the case where the codomain is reductive.

\begin{proof}[Proof of Theorem \ref{thm:margulis:fisher}]
Again by Theorem \ref{thm:restriction:similar}, we have a similarity between $G \ltimes X$ and the restricted groupoid $\mathcal{G}$. We denote by $\psi:G \ltimes X \rightarrow \mathcal{G}$ and by $\iota:\mathcal{G} \rightarrow G \ltimes X$ the two morphisms realizing the similarity. By Proposition \ref{prop:zariski:dense:similarity:invariant} we have a bijection
$$
\Psi:\mathrm{Hom}_{ZD}[\mathcal{G},H] \rightarrow \mathrm{Hom}_{ZD}[G \ltimes X, H].
$$
As a consequence
$$
\rho \circ \psi: G \ltimes X \rightarrow H
$$
is a Zariski dense measurable cocycle. We can apply Margulis-Fisher rigidity statement \cite[Theorem 3.16]{margulis:fisher} to obtain a representation $\pi:G \rightarrow H$ and a measurable cocycle $\sigma:G \ltimes X \rightarrow K \subset H$, where $K$ is a compact group centralizing $\pi(G)$, such that $\rho \circ \psi$ is similar to the product $\pi \cdot \sigma$. We have that $\rho$ is similar to $(\pi \cdot \sigma) \circ \iota$, as claimed. 
\end{proof}

\bibliographystyle{amsalpha}
\bibliography{biblionote}

\end{document}